\documentclass[preprint,12pt]{amsart}

\usepackage{fullpage}

\def\Z{{\mathbb{Z}}}
\def\Q{{\mathbb{Q}}}

\def\N{{\mathbb{N}}}
\def\C{{\mathbb{C}}}
\def\caL{{\mathcal{L}}}
\def\cO{{\mathcal{O}}}

\def\ra{{\rightarrow}}
\usepackage{amsmath}
\usepackage{amsfonts}
\usepackage{amssymb}
\usepackage{xypic}
\usepackage{longtable}
\usepackage{amsthm}
\usepackage{booktabs}
\usepackage{caption}
\usepackage{color}

\newcommand{\tc}{\tilde{c}}

\theoremstyle{plain}
\newtheorem{theorem}{Theorem}[section]
\newtheorem*{thm}{Theorem}
\newtheorem{corollary}[theorem]{Corollary}
\newtheorem{lemma}[theorem]{Lemma}
\newtheorem{proposition}[theorem]{Proposition}
\theoremstyle{definition}

\newtheorem{example}[theorem]{Example}
\newtheorem{question}[theorem]{Question}
\newtheorem{rem}[theorem]{Remark}
\newtheorem*{quest}{Question}

\DeclareMathOperator{\lcm}{lcm}

\makeatletter
\def\ps@pprintTitle{%
  \let\@oddhead\@empty
  \let\@evenhead\@empty
  \let\@oddfoot\@empty
  \let\@evenfoot\@oddfoot
}
\makeatother

\title[Non-simple surfaces and genus 3 curves]{Non-simple polarised abelian surfaces and genus 3 curves with completely decomposable Jacobians}
\author{Robert Auffarth and Pawe\l{} Bor\'owka}
\address{R. Auffarth \\Departamento de Matem\'aticas, Facultad de
Ciencias, Universidad de Chile, Santiago\\Chile}
\email{rfauffar@uchile.cl}
\address{P. Bor\'owka\\ Institute of Mathematics, Jagiellonian University in Krak\'ow\\Poland}
\email{pawel.borowka@uj.edu.pl}


\begin{document}

\begin{abstract}
We study the space of non-simple polarised abelian surfaces. Specifically, we describe for which pairs $(m,n)$ the locus of polarised abelian surfaces of type $(1,d)$ that contain two complementary elliptic curve of exponents $m,n$, denoted $\mathcal{E}_d(m,n)$ is non-empty. We show that if $d$ is square-free, the locus $\mathcal{E}_d(m,n)$ is an irreducible surface (if non-empty). We also show that the loci $\mathcal{E}_d(d,d)$ can have many components if $d$ is an odd square. As an application, we show that for a genus $3$ curve with a completely decomposable Jacobian (i.e. isogenous to a product of 3 elliptic curves) the degrees of complementary coverings $f_i:C\rightarrow E_i,\ i=1,2,3$ satisfy $\lcm(\deg(f_1),\deg(f_2))=\lcm(\deg(f_1),\deg(f_3))=\lcm(\deg(f_2),\deg(f_3))$.
\end{abstract}

\maketitle

\section{Introduction}
The moduli space of non-simple principally polarised complex abelian varieties is well understood. Indeed, the space of all non-simple principally polarised abelian varieties that contain an abelian subvariety of fixed dimension and exponent is irreducible and can be described as the quotient of a product space of polarised abelian varieties (see for instance \cite{Deb}, \cite{Auff}, \cite{Bor}). 
In particular, for principally polarised surfaces, exponents of complementary elliptic curves coincide and the periods of such surfaces have to satisfy so called Humbert singular equations (see for example \cite{ALR, Bor, vdG}). 

On the other hand, the moduli space of polarised abelian varieties of a given (non-principal) type is much harder to deal with. The main obstacle to extend the results of the principally polarised case is the fact that two complementary abelian subvarieties of a non-principally polarised abelian variety do not necessarily have the same exponent. This inevitably leads to the conclusion that the moduli space of non-simple polarised abelian varieties that contain an abelian subvariety of a fixed dimension is not necessarily irreducible, and it is not clear a priori how to describe the irreducible components.

In \cite[Ch. IX]{vdG}, G. van der Geer recalls Humbert's singular equations and states that they can be generalised to non-principally polarised surfaces in \cite[Th. IX.2.4]{vdG}.
However the result only states that there are many irreducible components of non-simple surfaces but does not give any insight into what the exponents of the elliptic curves are.


In \cite{G}, L. Guerra constructs moduli spaces of non-simple abelian varieties and especially their irreducible components. However, due to the fact that the construction is general and not very explicit, the author concludes his paper with questions about whether the constructed spaces are actually non-empty \cite[page 334]{G}. 

In \cite{ALR} the authors have found equations for non-simple polarised varieties. However, since they are defined over $\mathbb{Q}$ they do not provide any information about irreducible components and possible exponents of subvarieties.

The aim of this paper is to understand the case of non-simple polarised abelian surfaces. Let $\mathcal{A}_2(d)$ be a moduli of $(1,d)$ polarised abelian surfaces. Denote by 
\begin{align*}
\mathcal{E}_d(m,n):=\left\{(A,\mathcal{L})\in\mathcal{A}_2(d):\begin{array}{l}A\text{ contains a pair of complementary}\\
\text{ elliptic curves of exponents }m,n\end{array}\right\}.
\end{align*}

The first main result of the paper 
can be summarised in the following theorem
\begin{thm}[see Theorem \ref{irreducible}]
For $d,m,n\in\Z_+$, the moduli space $\mathcal{E}_d(m,n)$ is a non-empty (2-dimensional) subvariety of $\mathcal{A}_2(d)$ if and only if
\begin{equation*}\lcm(m,n)\cdot \gcd(m,n,d)=\gcd(m,n)d.\end{equation*}
\end{thm}
The idea of the proof is to use certain cyclic isogenies from principally polarised abelian surfaces in order to reduce the problem to working in the principally polarised case. Since we understand the situation for principally polarised surfaces, we need to carefully compute the degrees of the induced restrictions of the isogeny to each elliptic curve. To show the irreducibility (for a square free $d$) we consider modular curves as moduli spaces of elliptic curves with level structure and show that our objects of interest are equivalent under the action of the symplectic group.

However, if there is an odd square that is a divisor of $d$, then we are able to show that there are multiple components of $\mathcal{E}_d(d,d)$. We show the following result.
\begin{thm}[see Theorem \ref{thmmany}]
If $p^2|d$ for some odd $d$, then the locus $\mathcal{E}_d(d,d)$ has at least $\sigma(p)$ irreducible components, where $\sigma(p)$ is the number of divisors of $p$.
\end{thm}
The idea to distinguish irreducible components is to consider the subgroup $E\cap F$ and show that they may be parametrised by distinct finite abelian groups.

We conclude this part of the paper with a natural question.
\begin{quest}[see Question \ref{qirr}]
Does the number of irreducible components of $\mathcal{E}_d(d,d)$ equal $\sigma(p)$, where $p$ is the maximal odd number satisfying $p^2|d$?  
\end{quest}

The next part of the paper involves finding a ``normalised" equation in the Siegel space of periods of abelian varieties, similar to Humbert equations. We prove the following theorem.
\begin{thm}[see Theorem \ref{periods}]
 For any $d,m,n$, we can do the following construction:
 \begin{itemize}
     \item find coprime $a,b,c$ such that $m=\frac{cd}{a},\ n=\frac{cd}{b}$,
     \item find $u,v$ such that $au-bv=c$,
     \item define $x=\frac{du}{b},\ y=\frac{dv}{a}$ (they automatically satisfy $xy|d$).
 \end{itemize}
 Then 
 the set of matrices satisfying $z_3=(x+y)z_2-xyz_1$ with $\Delta=(\frac{cd}{ab})^2$ maps into an irreducible component of $\mathcal{E}_d(m,n)$. 
\end{thm}


As an application of the investigation of non-simple polarised abelian surfaces, we consider completely decomposable Jacobians of genus 3 curves. We show the following theorem.
\begin{thm}[see Theorem \ref{genus3}]
Let $C$ be a genus 3 curve such that $JC=E_1+E_2+E_3$ with $E_1,E_2,E_3$ complementary abelian subvarieties (so each elliptic curve $E_i$ induces a unique symmetric idempotent $\iota_i\in\mathrm{End}_\mathbb{Q}(JC)$ and $\iota_1+\iota_2+\iota_3=\mathrm{id}$). Then there exist maps $f_i:C\rightarrow E_i,\ i=1,2,3$ that satisfy $$\lcm(\deg(f_1),\deg(f_2))=\lcm(\deg(f_1),\deg(f_3))=\lcm(\deg(f_2),\deg(f_3)).$$
Moreover, if $d_1,d_2,d_3$ satisfy the above equations then there exists a 3--dimensional family of curves $C$ with maps $f_i$ satisfying $\deg(f_i)=d_i$.
\end{thm}

The structure of the paper is the following: In Section \ref{prelim} we go over some preliminaries of abelian surfaces and find important conditions on the exponents of elliptic curves, in Section \ref{grouptheoryresults} we prove crucial results on finite abelian groups with symplectic action that are applied in Section \ref{moduli} to prove the non-emptiness of the moduli spaces we are studying. In Section \ref{manycomp} we show when $\mathcal{E}_d(m,n)$ is irreducible and when the lower bound of components is large and in Section \ref{eqations} we show some equations for $\mathcal{E}_d(m,n)$. Finally in Section \ref{appl} we apply the results to genus 3 curves.

\subsubsection*{Acknowledgements}
The work of the second author has been partially supported by the Polish National Science Center project number 2018/30/E/ST1/00530. Some results have been achieved during his stay in Chile. He would like to thank University of Chile for hospitality. He also wishes to thank his wife Aleksandra for some valuable discussions on proofs in Section \ref{grouptheoryresults}.

\section{Preliminaries on abelian surfaces}\label{prelim}
Let $(A,\caL)$ be a complex polarised abelian surface of type $D:=(1,d)$. That is, $\caL$ is an ample line bundle and the map
\[\varphi_{\caL}:A\to A^\vee=:\mathrm{Pic}^0(A)\]
\[x\mapsto t_x^*\caL\otimes\caL^{-1}\]
is an isogeny with kernel isomorphic to $(\Z/d\Z)^2$. Moreover let $\mathcal{A}_2(d)$ denote the space of polarised abelian surfaces of type $(1,d)$. If $d=1$ we will simply use the standard notation $\mathcal{A}_2$. We wish to study the moduli space
\[\mathcal{E}_d(n):=\{(A,\mathcal{L})\in\mathcal{A}_2(d):A\text{ contains an elliptic curve of exponent }n\},\]
where the exponent of an elliptic curve $E\subseteq A$ is by definition the degree of $\caL|_{E}$, and in particular we want to understand what its irreducible components are.

If $E\subseteq A$ is an elliptic curve, then it corresponds to the image of a symmetric idempotent $\epsilon_E\in\mathrm{End}(A)\otimes\Q$ (where symmetric means that it is fixed by the Rosati involution given by $\caL$). We can define the \textit{complementary elliptic curve} of $E$ with respect to $\caL$ as 
\[F:=\mathrm{Im}(1-\epsilon_E),\]
where $\epsilon_F:=1-\epsilon_E$ is also a symmetric idempotent. We have that $E\cap F$ is finite and $A=E+F$, and therefore the natural addition map
\[E\times F\to A\]
is an isogeny.

We define another moduli space \begin{align*}
\mathcal{E}_d(m,n):=\left\{(A,\mathcal{L})\in\mathcal{A}_2(d):\begin{array}{l}A\text{ contains a pair of complementary}\\
\text{ elliptic curves of exponents }m,n\end{array}\right\}.
\end{align*}
Note that for a general pair $(m,n)$ the space $\mathcal{E}_d(m,n)$ will be empty, and that
\[\mathcal{E}_d(n)=\bigcup_{m}\mathcal{E}_d(m,n).\]


\begin{rem}\label{remproduct}
It is a classical result (see \cite[Cor 12.1.2]{BL}) that if $d=1$ then $m=n$. In the non--principally polarised case the complementary polarisation types can be quite unexpected. For $a,b\in\mathbb{N}$, consider a product $E\times F$ with the product polarisation $\cO_E(a)\boxtimes\cO_F(b)$. The kernel of the isogeny associated to the polarisation is $\mathbb{Z}_{a}^2\times\mathbb{Z}_b^2$, and by elementary group theory, this group is isomorphic to $\mathbb{Z}_{gcd(a,b)}\times\mathbb{Z}_{lcm(a,b)}$. In particular, the product polarisation is of type $(gcd(a,b),lcm(a,b))$, and so $E\times F$ can naturally be seen to belong to the moduli space $\mathcal{A}_2\left(d\right)$ where $d=lcm(a,b)/gcd(a,b)$. Moreover, with this primitive polarisation $E$ has exponent $a/gcd(a,b)$ and $F$ has exponent $b/gcd(a,b)$, so they are coprime.
\end{rem}

Going back to the original problem, we would like to find conditions on the pair $(n,m)$ for it to correspond to complementary polarisation types on a polarised abelian surface of type $(1,d)$. In other words, we want to know when $\mathcal{E}_d(m,n)\neq\varnothing$. An obvious necessary condition that follows from \cite[Cor 2.4.6.d]{BL} is that $d|lcm(m,n)$.

Write $A=\mathbb{C}^2/\Lambda$ where $\Lambda$ is a lattice in $\mathbb{C}^2$. If $\{\lambda_1,\lambda_2,\mu_1,\mu_2\}$ is a symplectic basis for $\Lambda$, then we see that the sublattice given by $\langle d\lambda_1,\lambda_2,\mu_1,\mu_2\rangle$ gives an abelian variety $X$ and a cyclic isogeny 
\[f:X\to A\] 
of degree $d$. Moreover, we see that $f^*\caL\equiv d\theta$ where $\theta$ is a principal polarisation and $\equiv$ stands for numerical equivalence. Let 
\[E':=f^{-1}(E)^0,\hspace{0.5cm}F':=f^{-1}(F)^0\]
be the corresponding elliptic curves on $X$. By \cite[Corollary 2.5]{ALR}, we have that $E'$ and $F'$ are complementary elliptic curves on $X$ with respect to $d\theta$, and therefore with respect to $\theta$. In particular, since $\theta$ is a principal polarisation, we have that $E'$ and $F'$ have the same exponent which we will denote by $c$. 

Define \[a:=|E'\cap\ker f|=\deg f|_{E'},\hspace{0.5cm }b:=|F'\cap \ker f|=\deg f|_{F'}.\]

Now, we are ready to show the necessary condition for the exponents $m,n$.
\begin{proposition}\label{propmainEF}
Let $A$ be $(1,d)$ polarised abelian surface that is non-simple. Let $E,F$ be a pair of complementary elliptic curves with exponents $n$ and $m$, respectively. Then $m$, $n$ and $d$ satisfy the following equation:
\begin{equation}\label{sufficient}lcm(m,n)\cdot gcd(m,n,d)=gcd(m,n)d.\end{equation}
\end{proposition}
\begin{proof}
Let $f:X\to A$ be a cyclic isogeny with $ker(f)=<P>$ and let $E'=f^{-1}(E)^0, F'=f^{-1}(F)^0$.
Assume $E'$ is of exponent $c$ in $X$. Then $E'\cap F'=E'[c]=F'[c]$.
Let $a=|E'\cap ker(f)|$, i.e $\frac daP$ generates the intersection and let $b=|F'\cap ker(f)|$ (i.e. $\frac dbP$ generates the intersection).

Then we claim that $(a,b)=(a,c)=(b,c)$.
To show this, assume $(a,b)=q$. Then $\frac dq P$ is a multiple of both $\frac da P$ and $\frac db P$, so $\frac dq P\in E'[q]\cap F'[q]$ hence $q|c$, so $q|(a,c)$.
On the other hand, if $(a,c)=r$ then $\frac dr P\in ker(f)\cap E'[r]$, so  $\frac dr P\in ker(f)\cap F'[r]$ since $E'[c]=F'[c]$, and hence $r|b$ and $r|(a,b)$. An analogous argument shows that $(a,b)=(b,c)$. This shows that indeed $(a,b)=(a,c)=(b,c)=q$ for some $q|d$.


Note that $exp(E)=\frac{cd}{a}$ and $exp(F)=\frac{cd}{b}$. Since $gcd(\frac{cd}{a}, \frac{cd}{b})=\frac{cdq}{ab}$ and $gcd(\frac{cd}{a}, \frac{cd}{b},d)=\frac{dq^2}{ab}$ both sides of Equation \ref{sufficient} equal $\frac{cd^2q}{ab}$.
\end{proof}
\begin{rem}\label{remcoprime}
In the previous proof, we have found a possible relation $(a,b)=(a,c)=(b,c)=q$ for some $q|d$.
It is worth to note that although $q$ may vary, the exponents $m,n$ do not change.
This is an indicator that $\mathcal{E}_d(m,n)$ may have multiple irreducible components.

Since our aim is to show that $\mathcal{E}_d(m,n)$ is non-empty, we assume from now on that $q=1$ and hence $a,b,c$ are pairwise coprime. We will comment on other cases in Section \ref{manycomp}.
\end{rem}

\subsection{A number of possible complementary exponents}
We define the function
\[\mu:\mathbb{N}^2\to\mathbb{N}\]
as follows: If $n,d\in\mathbb{N}$ we can write $n=\tilde{n}t$ and $d=\tilde{d}t$ where $t$ is the maximal integer such that $gcd(\tilde{n},t)=gcd(\tilde{d},t)=1$. Then we define  
\[\mu(n,d):=\sigma(t),\]
where $\sigma$ is the function that counts the number of positive divisors of a number.

\begin{lemma}\label{phind}
If $n,d\in\mathbb{N}$, then the number of $m\in\mathbb{N}$ that satisfy Equation \eqref{sufficient} is $\mu(n,d)$.
\end{lemma}
\begin{proof}
Write 
\[n=\prod_{i=1}^rp_i^{\alpha_i},\hspace{0.5cm}d=\prod_{i=1}^rp_i^{\beta_i},\hspace{0.5cm}m=\prod_{i=1}^rp_i^{\gamma_i}\]
where $\alpha_i,\beta_i,\gamma_i\geq0$ and $\beta_i\leq\max\{\alpha_i,\gamma_i\}$. Now Equation \eqref{sufficient} turns into
\[\alpha_i+\gamma_i+\min\{\alpha_i,\beta_i,\gamma_i\}=2\min\{\alpha_i,\gamma_i\}+\beta_i.\]
If $\alpha_i=\beta_i$ (only when $p\mid t$), we get the equation
\[\gamma_i+\min\{\alpha_i,\gamma_i\}=2\min\{\alpha_i,\gamma_i\},\]
which is satisfied if and only if $\gamma_i\leq \alpha_i$. Therefore $\gamma_i$ can take on $\alpha_i+1$ possibilities.

If $\alpha_i>\beta_i$, then we get the equation
\[\alpha_i+\min\{\beta_i,\gamma_i\}=\gamma_i+\beta_i.\]
This can only be satisfied if $\gamma_i=\alpha_i$ and so there is one possibility.

If $\alpha_i<\beta_i$, then we get the equation
\[\alpha_i+\gamma_i=\min\{\alpha_i,\gamma_i\}+\beta_i\]
which has the unique solution $\gamma_i=\beta_i$. By counting the total number of possibilities for each $i$, we get the number $\mu(n,d)$.
\end{proof}

The aim of the rest of the paper is to show that the restriction of Proposition \ref{propmainEF} is in fact sufficient for the existence of elliptic curves with the given exponents. 
 Finally, as a consequence Lemma \ref{phind} computes the number of possible complementary exponents for $\mathcal{E}_d(n)$.

\section{Symplectic modules over $\Z_d$}\label{grouptheoryresults}
Now we need some results from the theory of abelian groups. Let $d>1$ be a positive integer. Our main object of interest is the group $\Z_d^4$ equipped with a non-degenerate symplectic form $\omega$ over $\Z_d$ (i.e. $\omega$ is bilinear, alternating and $\omega(x,\cdot)=0$ if and only if $x=0$). We will use the perspective that $\Z_d^4$ is a free module over $\Z_d$. 
The results can be stated in a more general setting, but it is not necessary for our purposes.

We start with a useful fact that we will use quite often:
\begin{lemma}\label{lemdivision}
Let $k|d$ and let $x\in\Z_d^n$ be of order $k$. Then there exists $y\in\Z_d^n$ of order $d$ such that $x=\frac dk y$. 
\end{lemma}
\begin{proof}
Let $x$ be of order $k$ and consider the canonical projection $\pi\colon \Z_d^n\to\Z_d^n/\langle x\rangle$. 
Note that by the fundamental theorem of finite abelian groups $\Z_d^n/\langle x\rangle\cong\Z_d^{n-1}\oplus \Z_{\frac{d}{k}}$. Now, $\pi^{-1}(0\oplus\Z_{\frac{d}{k}})$ is a cyclic group of order $d$ that contains $x$. Hence, there exists a generator $y$ such that $x=\frac{d}{k}y$.
\end{proof}
As usual, for $k\in\mathbb{N}, E\subset\Z_d^n$ we denote $kE=\{kx: x\in E\}$ and $E[k]=\{x\in E:kx=0\}$. Then, the lemma yields an immediate corollary.
\begin{corollary}
If $E\cong \Z_d^2\subset\Z_d^4$ and $k\mid d$, then $E[k]=\frac dk E$.
\end{corollary}
\begin{proof}
The inclusion $\frac{d}{k}E\subseteq E[k]$ is obvious. If $x\in E[k]$ is of order $\ell$ (and so $\ell\mid k$), then the previous lemma tells us that $x=\frac{d}{\ell}y$ for some $y$ of order $d$. Now 
\[x=\frac{d}{k}\frac{k}{\ell}y\]
which completes the proof. 
\end{proof}
The first result deals with the case of a direct sum.
\begin{lemma}\label{technical0}
For an integer $d>1$, let $E\cong F\cong \Z_d^2$ with the canonical bases $e_1,e_2$ and $f_1,f_2$ be equipped with the alternating forms given by $\omega_E(e_1,e_2)=\omega_F(f_1,f_2)=1$. Consider $\Z_d^4=E\oplus F$ with the direct sum of forms $\omega_E+\omega_F$. Let $a,b$ be coprime divisors of $d$.
Then, there exists a cyclic subgroup $G$ of order $d$ such that $|G\cap E|=a$ and $|G\cap F|=b$. Moreover, all such groups are equivalent under the action of a (symplectic) group $Sp(E,\omega_E)\times Sp(F,\omega_F)$.
\end{lemma}
\begin{proof}
Firstly, consider $P=be_1+af_1$. Since $e_1,f_1$ are independent, for $k\in\Z_d$ we have $kP=0$ if and only if $kb=ka=0\ mod\ d$. Since $a,b$ are coprime, then $lcm(\frac da,\frac db)=d$, so $d|k$. Hence $P$ is of order $d$.
Now, we need to prove that $G=\langle P\rangle$ satisfies the assertion. Firstly, note that $kP\in E$ if and only if $\frac{d}{a}|k$. Hence, we can write $k=\frac{d}{a}l$, for $l=1,\ldots,a$. There are exactly $a$ possible  values of $l$ and $\frac{ld}{a}e_1=0$ if and only if $l=a$. This shows $|G\cap E|=a$. Completely analogously, one can show that $|G\cap F|=b$.

Let $H$ be another such group. Let $Q=k_1e_1+k_2e_2+zf_1+tf_2$ be a generator of $H$. Since $H\cap E$ is of order $a$, we know that $\frac{d}{a}z=\frac{d}{a}t=0$. Hence $z=az',\ t=at'$. for some $z',t'$. Let $f_1'=z'f_1+t'f_2$. Since $Q$ is of order $d$ and therefore $zf_1+tf_2$ is of order $\frac da$, we get that $f_1'$ is of order $d$. Let $f_2'\in F$ be chosen so that $\omega_F(f_1',f_2')=1$. Now, for a symplectic basis $e_1,e_2,f'_1,f'_2$ we have $Q=k_1e_1+k_2e_2+af_1'$. Analogously, one can find symplectic $e_1',e_2'\in E$ such that $Q=be_1'+af_1'$. Since both bases are symplectic, the transition matrix belongs to $Sp(E,\omega_E)\times Sp(F,\omega_F)$ and hence $H$ is equivalent to $G$.
\end{proof}

\begin{rem}
For a non-prime integer $d$, having $E\cong\Z_d^2, E\subset\Z_d^4$ does not imply the existence of $F$ such that $\Z_d^4=E\oplus F$. Therefore, we need a more general and more technical construction that will involve working over $\Z_{cd}$, for some $c>1$.
\end{rem}

For an integer $d>1$, let $a,b,c$ be pairwise coprime integers, such that $a,b$ are divisors of $d$ and $c>1$. Let $E\cong F\cong \Z_{cd}^2$ with the canonical bases $e_1,e_2$ and $f_1,f_2$ be equipped with the alternating forms given by $\omega_E(e_1,e_2)=\omega_F(f_1,f_2)=1$. Consider $\Z_d^4=E\oplus F$ with the direct sum of forms $\omega_E+\omega_F$. Assume, there is an isotropic subgroup $K\cong \Z_c^2$ such that $K\cap E=K\cap F=\{0\}$ and let $X=(E\oplus F)/K$ with $\pi:E\oplus F \to X$ be the quotient projection. Such group is an example of so called allowable isotropic subgroup in \cite{Bor}. 

Note that since $K\subset (E\oplus F)[c]$ we can write every element of $K$ as $dP$ and instead of using restricted symplectic form (that may be trivial), we use the induced symplectic form that is given by $\omega_K(dP,dQ)=\omega(P,Q)\ mod\ c$ for which $K$ has to be isotropic, too.

We will start with the following lemma.
\begin{lemma}\label{anyK}
For any $K$ satisfying the condition above there exists a symplectic basis $e_1,e_2\in E, f_1,f_2\in F$ such that $K=\langle dbe_1+daf_1, dae_2-dbf_2\rangle$.
\end{lemma}
\begin{proof}
Firstly, note that $K=\langle dbe_1+daf_1, dae_2-dbf_2\rangle $ satisfies the above properties. Certainly $K\subset (E\oplus F)[c]$. Moreover, since $a,b$ are coprime, both generators are of order exactly $c$ and they are linearly independent. $K$ is obviously isotropic since $ab=ba$. Moreover, since $a,c$ (and $b,c$) are coprime, elements $dbe_1,dae_2$ generates $E[c]$ and hence $K\cap E=\{0\}$. Analogously $K\cap F=\{0\}$. 

Now, consider the multiplication map $f_{bd}:\Z_{cd}\to \Z_c$. It is surjective, because $b,c$ are coprime. Moreover, for any generator $x\in \Z_c$ there exists a generator $y\in\Z_{cd}$ such that $x=bdy$. To see this, one can use Dirichlet's theorem on arithmetic progressions to see that $y_k=y+kc$ has infinitely many primes, hence at least one is coprime to $d$ and hence its class in $\Z_{cd}$ is a generator. The map $f_{ad}$ has the same properties. 

Now, let $x,y\in K$ be generators of $K$. By what we have proved, we can write $x=dbe_1+daf_1$ and $y=dae_2'-dbf_2'$ for some elements $e_1, e_2'\in E, f_1,f_2'\in F$ of order $cd$. By the fact that $K\cap E=\{0\}$ we see that $e_1,e_2'$ are independent and hence generate $E$. Therefore $\omega(e_1,e_2')=\alpha$ is a unit in $\Z_{cd}$. Note that $K$ is isotropic so by direct calculation of $\omega_K(x,y)$ one gets that $\omega(f_1,f_2')=\alpha\ mod \ c$, too. Hence, by a possible change of $f_2'$, we can assume $\omega(f_1,f_2')=\alpha$ in $\Z_{cd}$. 

Now, let $e_2=\alpha^{-1}e_2', f_2=\alpha^{-1}f_2'$. Then $e_1,e_2, f_1,f_2$ is a symplectic basis and $x, \alpha^{-1}y$ are generators of $K$ of the desired form.
\end{proof}

\begin{lemma}
Let $E'=\pi(E\times\{0\}), F'=\pi(\{0\}\times F)$. 
Then $E'\cap F'=\pi(dE\times \{0\})=\pi(\{0\}\times dF)$. 
In particular, we have that $\pi^{-1}(E'\cap F')=dE+dF$.
\end{lemma}
\begin{proof}
Firstly, note that for every $dx\in dE$ there exists a unique $dy\in dF$ such that $dx+dy\in K$ and all elements of $K$ are of this form. 

Hence, for any $P'=Q'\in E'\cap F'$, we have $P'-Q'=0\in X$, so there exist $P\in E$ with $\pi(P)=P'$ and $Q\in F$ with $\pi(Q)=Q'$ such that $P-Q\in K$. Hence $P-Q=dx+dy$ so $P-dx=Q+dy=0$. hence $P'=\pi(dx+0)=Q'=\pi(0-dy)$.

On the other hand if $P'=\pi(dx+0)$ then there exist $dy\in dF$, such that $dx+dy\in K$, so $P'=\pi(dx+0)=\pi(0-dy)=Q'\in E'\cap F'$.

Note that by comparing the orders, one gets $\pi^{-1}(E'\cap F')=dE+dF$.
\end{proof}

Let $X[d]$ be the set of $d$-torsion points of $X=(E\oplus F)/K$. Note that $X[d]\cong\Z_d^4$. Now, we are interested in the set $\pi^{-1}(X[d])$ that will be denoted by $L$. From the definition, we can write $L=\{e+f: e\in E, f\in F, de+df\in K\}$. By $l$, we denote $l=gcd(c,d)$. Let $\tilde{c}=\frac cl$.
Note that $cd=\tc ld$.

\begin{lemma}
The set of cyclic groups $G'\subset X[d]$ of order $d$ such that $G'\cap E'\cap F'=\{0\}$ is 
dominated by the set of
cyclic subgroups $\tilde{G}\subset L$ of order $ld$ (via the map $\pi$).
\end{lemma}
\begin{proof}
If $P\in \tilde{G}$ is a generator of order $ld$ then $kP\in dE+dF$ only for $k$ such that $d|k$, so $\pi(\tilde{G})$ is a desired group.
On the other hand, if $|\tilde{G}|< ld$, hence $P\in \tilde{G}$ is a generator of order less than $ld$ then we can write $P=k(e+f)$ for some generators $e\in E$ and $f\in F$ and $k>1, k|d$. Hence, one gets $\frac d{k}P\in dE+dF$. Now, if $\pi(\tilde{G})$ is of order $d$, then $\frac d{k} \pi(P)\neq 0$ and we get that $\pi(\tilde{G})\cap E'\cap F'\neq \{0\}$.
\end{proof}

\begin{lemma}\label{technical1}
For any $K$ as above and $G$ that satisfy $|G|=ld,\ |G\cap K|=l,\ G\cap (dE+dF)=\{0\}$ and $|G\cap E|=a, |G\cap F|=b$ there exists a symplectic basis $e_1,e_2\in E,f_1,f_2\in F$ such that $G=\langle \tilde{c}be_1+\tilde{c}af_1\rangle,\ K=\langle dbe_1+daf_1, dae_2-dbf_2\rangle$.
\end{lemma}
\begin{proof}
Let $u\in K$ be of order $c$, such that $\tc u\in G$ and let $v\in G$ be a generator of $G$ (of order $ld$, such that $dv\in K$). We claim that $u+v$ is of order $cd$. To see this, let $k(u+v)=0$ for the smallest positive $k$. This means that $ku=-kv\in G\cap K$. In order to have $ku\in G$ one needs $\tc|k$ and for $kv\in K$ one needs $d|k$. Since $\tc,d$ are coprime we have $\tc d|k$. Now, note that $c|\tc d$, so in fact, $ku=0$ and therefore $ld|k$.
Since $lcm(\tc,ld)=cd$ we get that $cd|k$ and so $u+v$ is of order $cd$.

Note that $d(u+v)\in K$ is of order $c$ and $\tc (u+v)\in G$ is a generator of $G$.

By the proof of Lemma \ref{technical0} (applied for order `ld'), we can write $v=\tc be_v+\tc af_v$ and by Lemma \ref{anyK} we can write $u=dbe_u+daf_u$ for some $e_u,e_v\in E, f_u,f_v\in F$.

Now, $u+v=be_1'+af_1'$ for $e_1'=de_u+\tc e_v, f_1'=df_u+\tc f_v$.
It may happen that $f_1'$ is not of order $cd$. However, since $|G\cap E|=a$ and is given by $\frac{cd}{|\langle af_1'\rangle|}$, we get that $af_1'$ is of order $\frac{cd}{a}$. By Lemma \ref{lemdivision}, we get that $af_1'=af_1$ for some $f_1$ of order $cd$. Analogously, we can find $e_1$ of order $cd$ such that $be_1'=be_1$ and therefore $u+v=be_1+af_1$ for $e_1\in E,\ f_1\in F$ of order $cd$. 
By the proof of Lemma \ref{anyK} one can extend $e_1,f_1$ symplectically to $e_2,f_2$ such that $K=\langle d(u+v), dae_2-dbf_2\rangle$, which proves the assertion.
\end{proof}

\begin{rem}
Note that, a posteriori, Lemma \ref{technical0} can be seen as a special case of Lemma \ref{technical1} when $c=\tc=1$, $K=\{0\}$ and $X=E\oplus F$.
\end{rem}

\section{Non-emptiness of $\mathcal{E}_d(m,n)$}\label{moduli}
We would like to describe the moduli of non-simple abelian surfaces $\mathcal{E}_d$. Our natural starting point is a moduli space of products of elliptic curves with a chosen $d$-torsion point which in the literature are sometimes called abelian surfaces with a root (cf. \cite{BLiso,HW}). 

Since a non-simple principally polarised surface is only isogenous to a product, we also use a construction from \cite{Bor} that involves so-called allowable isotropic subgroups. 

In order to make this precise, for $N\in\mathbb{N}$, let $X(N)$ denote the modular curve associated to the congruence subgroup $\Gamma(N):=\ker(\mathrm{SL}(2,\mathbb{Z})\to\mathrm{SL}(2,\mathbb{Z}_N))$, i.e. the moduli space of triples $(E,e_1,e_2)$ where $e_1,e_2\in E$ is a symplectic basis of torsion points of order $N$. 

Now let $a,b,c\in\mathbb{N}$ be coprime integers with $a|d$, $b|d$ and $\tc=\frac{c}{\gcd(c,d)}$ and consider the moduli map
\[\Phi_{a,b,c}:X(cd)\times X(cd)\to\mathcal{A}_2(d)\]
\[((E,e_1,e_2),(F,f_1,f_2))\mapsto((E\times F)/\langle \tilde{c}be_1+\tilde{c}af_1, dbe_1+daf_1,  dae_2-dbf_2\rangle),H)\]
where $H$ is the polarisation that pulls back to $cd$ times the principal product polarisation on $E\times F$.                         
\begin{lemma}\label{phiabc}
The map $\Phi_{a,b,c}$ is a well-defined morphism and its image lies in $\mathcal{E}_d(\frac{cd}{a},\frac{cd}{b})$.
\end{lemma}
\begin{proof}
We start with $(E\times F, \mathcal{O}_E(cd)\boxtimes \mathcal{O}_F(cd))$ with a $(cd,cd)$ product polarisation.
We define $K=\langle dbe_1+daf_1,  dae_2-dbf_2\rangle$ and we see that $K$ is an allowable isotropic subgroup of the set of $c$-torsion points, so we can define a quotient $\pi_c: E\times F\to (E\times F)/K$ and by \cite[Cor 6.3.5]{BL} and Riemann-Roch we can compute that $X=(E\times F)/K$ has a $(d,d)$ polarisation and the quotient map $\pi_c$ is polarised. Now if $P=\tilde{c}be_1+\tilde{c}af_1$ then $\pi_c(P)\in X[d]$ and $A=X/\langle P\rangle$ is indeed (again by \cite[Cor 6.3.5]{BL}) a $(1,d)$-polarised abelian surface with the quotient map $\pi_d$ being a polarised isogeny. 

The last assertion follows from the fact that $\pi_c|_{E\times\{0\}}$ is injective, whereas $\pi_d|_{\pi_c(E\times\{0\})}$ is of order $a$ and analogously $(\pi_d\circ\pi_c)|_{\{0\}\times F}$ is of order $b$.
\end{proof}

\begin{theorem}\label{irreducible}
Let $m,n,d$ be positive integers that satisfy Equation \ref{sufficient}.
Then, one can write $m=\frac{cd}{a},\ n=\frac{cd}{b}$ for some $a,b$ coprime divisors of $d$ and $c$ coprime to $a,b$ and the image of  $\Phi_{a,b,c}$ is contained in $\mathcal{E}_d(m,n)$. This shows that $\mathcal{E}_d(m,n)$ contains an irreducible component of dimension 2, hence it is non-empty.
\end{theorem}
\begin{proof} 
By dividing Equation \ref{sufficient} by $gcd(m,n)$ one gets $lcm(m,n)|gcd(m,n)d$, hence by elementary number theory we can write $lcm(m,n)=gcd(m,n)\cdot s$ for some $s|d$ and so $m=gcd(m,n)\cdot b,\ n=gcd(m,n)\cdot a$ for some $a,b$ that are coprime and satisfy $ab=s$. In particular $a,b$ are divisors of $s$ and hence divisors of $d$. Substituting $m,n$ we get $gcd(m,n,d)=\frac{d}{ab}$, so $m=\frac{cd}{a},\ n=\frac{cd}{b}$ for $c$ coprime to $a,b$.
Moreover, the triple $(a,b,c)$ of pairwise coprime numbers is uniquely determined by $m,n$ and all such triples $(a,b,c)$ appear in this way.
Now, Lemma \ref{phiabc} shows that $\mathcal{E}_d(m,n)$ is non-empty.\end{proof}

 The following corollary computes possible complementary exponents if we fix a pair $d, n$.

\begin{corollary}\label{corposmn}
Let $n,d\in\N$. We can write uniquely $n=\frac{cd}{a}$, 
where $a=\frac{d}{gcd(n,d)}$ and $c=\frac{n}{gcd(d,n)}=\frac{lcm(n,d)}{d}$. Let $s$ be a divisor of $d$ such that $a|s$ and $a,\frac{s}{a}, c$ are pairwise coprime. Note that $s$ always exists (e.g. $s=a$). Denote by $b=\frac{s}{a}$.

Then 
$m=\frac{cd}{b}$ satisfies Theorem \ref{irreducible} and is a complementary type to $n$. In particular $\mathcal{E}_d(n)$ is non-empty and the number of possible complementary exponents equals $\mu(n,d)$ (see Lemma \ref{phind}).
\end{corollary}

\begin{rem}
In extreme cases, we have the following. If $d^2|n$ then the only possibility is $a=b=1$ and both exponents equal $n$.
If $(d,n)=1$, then $c=n,\ s=a=d$, so $b=1$ and we have exponents $n, dn$.
On the other hand, if $n=cd$ for some $c$ coprime to $d$ then $a=1$ and there are as many complementary types as divisors of $d$ (of the form $cb$ for $b|d$). \end{rem}


\subsubsection*{Polarised products of elliptic curves}
In the case of principally polarised abelian surfaces, the polarised product of two elliptic curves is a distinguished irreducible divisor that is the complement of the locus of Jacobians of smooth genus 2 curves.

We would like to describe the locus of polarised products of two elliptic curves also in the $(1,d)$ polarised case.

Corollary \cite[5.3.6]{BL} asserts that a polarised abelian surface $A$ is a polarised product of two complementary elliptic curves E,F if and only if the addition map $g:E\times F\to A$ is a (polarised) isomorphism.
 Then, Remark \ref{remproduct} yields the following corollary.
\begin{corollary}
A $(1,d)$ polarised abelian surface $A$ is a polarised product of two elliptic curves E,F of exponents $m,n$ if and only if $m,n$ are coprime and $mn=d$.

Note that every such pair $m,n$ satisfies Equation \ref{sufficient}.
\end{corollary}

Now, the loci $\{(E\times F, \cO_E(m)\boxtimes\cO_F(n)): E,F \text{ elliptic curves}\}$ are all irreducible components of the locus of products of elliptic curves. The number of such components is equal to $2^{\omega(d)-1}$ where $\omega(d)$ is the function counting all prime factors of $d$.

\section{Possible number of components}\label{manycomp}
Since we have already answered the question of when $\mathcal{E}_d(m,n)$ is non-empty, from now on, we will discard the assumption that $a,b,c$ are pairwise coprime.  

However, in many cases, the moduli space $\mathcal{E}_d(m,n)=im(\Phi_{a,b,c})$ is in fact irreducible. To show this, we start with the following lemma that may be called reduction lemma.
\begin{lemma}\label{lemredd}
Let $X$ be a principally polarised abelian variety, $E',F'\subset X$ complementary curves of exponent $c$ and let $f:X\to A$, be a cyclic isogeny of degree $d$. Assume $\ker f\cap E'\cap F'= \langle x\rangle$ for some point $x$. Then there exists another principally polarised abelian variety $Y$ with elliptic curves $E''=E'/\langle x\rangle$ and $F''=F'/\langle x\rangle$ and a map $h:Y\to A$ such that $\ker h\cap E''\cap F''= \{0\}$.
\end{lemma}
\begin{proof}
In the proof, we abuse notation by writing $E'$ both as an abstract curve and its image in $X$. In the same way, $x$ is both a point of $X$ that is in the image of $E'$ and a point of $E'$.

For any map $g$, let $\overline{g}$ denote the induced map on the quotient. Let $q$ denote the quotient map $X\to X/\langle x\rangle$, and let $g$ be the addition map. We have the following diagram of polarised isogenies:
\begin{equation} \label{reductionshort} 
\xymatrix@R=1cm@C=1.5cm{
 E'\times F' \ar[r]^(.55){g} 
 & X \ar[r]^(.45){f} \ar[dr]_(.4){q}  
 &A 
 \\
&& X/\langle x\rangle  \ar[u]_(.4){\overline{f}}
}
\end{equation}
Since $\ker(g\circ q)$ contains the group generated by $(x,0)$ and $(0,x)$ in $E'\times F'$ we get the quotient map and the extended diagram
\begin{equation} \label{reduction} 
\xymatrix@R=1cm@C=1.5cm{
 E'\times F' \ar[r]^(.55){g} \ar[d]_{quotient}  & X \ar[r]^(.45){f} \ar[dr]_(.4){q}  
 &A 
 \\
E'/\langle x\rangle\times F'/\langle x\rangle
\ar[rr]^(.6){\overline{g\circ q}} &
& X/\langle x\rangle  \ar[u]_(.4){\overline{f}}
}
\end{equation}
The polarisation types are as follows, where $l$ is the order of $x$:
\begin{equation} \label{reductionpol} 
\xymatrix@R=1cm@C=1.5cm{
 (cd,cd) \ar[r]^(.55){g} \ar[d]_{quotient}  & (d,d) \ar[r]^(.45){f} \ar[dr]_(.4){q}  
 &(1,d) 
 \\
(\frac{c}{l}d,\frac {c}{l}d)  \ar[rr]^(.6){\overline{g\circ q}} &
& (\frac dl,d)  \ar[u]_(.4){\overline{f}}
}
\end{equation} Since $\overline{f}\circ\overline{g\circ q}$ is a polarised isogeny of order $\frac{c^2d}{l^2}$, we can factor it via the map $h'$ of order $\frac{c^2}{l^2}$ to a surface $(Y,d\theta_Y)$ and $h:Y\to A$ that is an isogeny of order $d$. Moreover, the restrictions of $h'$ to $E''$ and to $F''$ are injective, we get that $h'$ is addition and $E''$, $F''$ in $Y$ are elliptic curves that are complementary to each other.
\end{proof}
\begin{theorem}\label{squarefreeirr}
If $d$ is square free and $m,n,d$ satisfy Equation \ref{sufficient} then $\mathcal{E}_d(m,n)=im(\Phi_{a,b,c})$ is an irreducible surface in $\mathcal{A}_2(d)$.
\end{theorem}
\begin{proof}
Let $A\in\mathcal{E}_d(m,n)$. Let $f:X\ra A$ be a cyclic isogeny from a principally polarised abelian surface. By Lemma \ref{lemredd} we can construct an isogeny $h$ and since $d$ is square free, we can assume $h$ is cyclic. By construction $\ker h\cap E''\cap F''= \{0\}$ which means that $a,b$ are coprime, so $A\in im(\Phi_{a,b,c})$.
\end{proof}

\begin{corollary}
For a square-free positive integer $d$, the locus of non-simple polarised abelian surfaces in $\mathcal{A}_2(d)$ is equal to
\[\mathcal{E}_d=\bigcup_{m,n\in\mathbb{N}}\mathcal{E}_d(m,n)=\bigcup_{\stackrel{a|d,b|d,c}{(a,b)=(a,c)=(b,c)=1}}\mathrm{Im}(\Phi_{a,b,c}),\]
where the components are irreducible (or empty).
\end{corollary}

\subsection{Many components}
The reduction lemma does not assert that $h$ is cyclic. Now, we would like to provide an explicit example that shows that $\mathcal{E}_9(9,9)$ has at least two components.
We will use the following notation (that appears in Lemma \ref{lemxy}). Let $e_1,e_2$ denote canonical basis of $\C^2$ and $f_1,f_2$ be the first and the second column of the period matrix. Then  $\omega(e_1,f_1)=\omega(e_2,f_2)=1$ is the alternating form that defines a polarisation.
\begin{example}\label{09}
Let $A_{0,9}=\C^2/\Lambda$  with the lattice  $\Lambda$ generated by the columns of $$\left[\begin{array}{cccc}z_1&z_2&1&0  \\
    z_2 &9z_2&0&9 
\end{array}\right]$$
for $z_1,z_2$ such that the imaginary part of the period matrix is positive definite.

Lemma \ref{lemxy} shows that there exist elliptic curves $E_0,\ E_9$ that are embedded in $A_{0,9}$ via maps of analytic representations $[1\ 9]$ with the primitive lattice vectors $f_2, e_1+9e_2$ and $[1\ 0]$ with the primitive lattice vectors $9f_1-f_2, e_1$. In particular, exponents of both curves are equal to 9, so $A_{0,9}\in \mathcal{E}_9(9,9)$.
One can also compute that $E_0\cap E_9=<\frac{1}{9}f_2>\simeq\Z_9$.
\end{example}

\begin{example}\label{36}
Let $A_{3,6}=\C^2/\Lambda$  with the lattice  $\Lambda$ generated by the columns of $$\left[\begin{array}{cccc}z_1&z_2&1&0  \\
    z_2 &9z_2-18z_1&0&9 
\end{array}\right]$$
for $z_1,z_2$ such that the imaginary part of the period matrix is positive definite.

Now, we can explicitly compute that there exist elliptic curves $E_3,\ E_6$ that are embedded in $A_{3,6}$ via maps of analytic representations $[1\ 3]$ with the primitive lattice vectors $f_2-6f_1, 3e_1+9e_2$ and $[1\ 6]$ with the primitive lattice vectors $f_2-3f_1, 3e_1+18e_2$. In particular, exponents of both curves are equal to 9, so $A_{3,6}\in \mathcal{E}_9(9,9)$.
Note that $E_3\cap E_6=<3e_2,\frac{f_2}{3}>\simeq\Z_3\times\Z_3$.
\end{example}

\begin{corollary}
Examples \ref{09} and \ref{36} explicitly show that the locus $\mathcal{E}_9(9,9)$ has at least two irreducible components.
\end{corollary}
An explanation of this phenomenon is as follows. Example \ref{09} is constructed by taking the product $E\times F$ (with the polarisation type $(9,9)$) and dividing by an allowable cyclic subgroup of order $9$.

On the other hand, in Example \ref{09} one takes the product $E\times F$ (with the polarisation type $(9,9)$) and divides by an allowable $\Z_3\times\Z_3$ subgroup that is non-isotropic with respect to $(3,3)$ type but trivially isotropic for type $(9,9)$. Such a choice forces the quotient surface to be $(1,9)$ polarised.
One can check that if one wants to find a cyclic isogeny $f:X\ra A_{3,6}$ from a principally polarised $X$ then $X$ will contain elliptic curves of exponent 3 and reduction lemma will give you $h$ that will essentially be a quotient by $\Z_3\times\Z_3$ and hence it will never be cyclic in this case.

Now, we would like to generalise the above phenomenon in the following lemma.
\begin{lemma}\label{lemany}
Let $p=p_1\cdot\ldots\cdot p_k$ for some prime numbers $p_i>2$ and let us assume that $d$ is an odd number satisfying $p^2|d$.
Let $X=E\times F$. Then, for any divisor $k|p$ there exists an allowable subgroup $G$ that is isomorphic to $\Z_k\oplus \Z_{\frac{d}{k}}$ such that $X/G$ is of type $(1,d)$.
\end{lemma}
\begin{proof}
Let $e_i\in E, f_i\in F$ be such that $e_1,e_2, f_1, f_2\in X[d]$ is a symplectic basis of $X[d]$. Let $g_1=2e_1+f_1, g_2=e_2-f_2$. Since $d$ is odd, we see that $<g_1>\cap F=\{0\}$.
Now, let $G=<kg_1,\frac{d}{k}g_2>$. Since $g_1,g_2$ are linearly independent, we get that $G\simeq \Z_k\oplus \Z_{\frac{d}{k}}$ and $G\cap E=G\cap F=\{0\}$. Moreover, since $\omega(g_1,g_2)=1 \mod d$, we get that $\omega(kg_1, \frac{d}{k}g_2)=d$. In particular, $G$ is isotropic with respect to $(d, d)$ polarisation. Now, consider $(d',d')$ type and denote by $l=\frac{d}{d'}>1$. Then $le_1,le_2, lf_1, lf_2\in X[d']$ is a symplectic basis with $\omega'=\frac{1}{l^2}\omega$ being the symplectic form. Then, either $l\nmid k$ and in this case $kg_1\notin X[d']$, so $G$ is not a subgroup of the kernel of the polarising isogeny or $l\mid k$. In this case $\omega'(kg_1, \frac{d}{k}g_2)=\frac{d}{l^2}=\frac{d'}{l}\neq 0 \mod d'$, hence $G$ is not isotropic with respect to $(d',d')$ polarisation type.

The fact that $X/G$ is of type $(1,d)$ follows from \cite[Corollary 6.3.5]{BL}. Assuming by contradiction that $X/G$ is of type $(l,\frac{d}{l})$ for some $l>1$ we could have found a polarisation of type $(1,\frac{d}{l^2})$ on $X/G$ and pull it back to $X$ to get a polarisation of type $(\frac{d}{l},\frac{d}{l})$ for which $G$ would have to be isotropic.
\end{proof}

\begin{rem}
Note that we cannot extend the statement of Lemma \ref{lemany} to $p=2$ and $d=4$. It is because if $G\cong \Z_2\times \Z_2$ satisfies the condition $G\cap E[2]=G\cap F[2]=\{0\}$ then you can find a symplectic basis $e_1,e_2,f_1,f_2$ such that $G=<e_1+f_1,e_2+f_2>$ and hence $G$ is always isotropic with respect to $(2,2)$ polarisation type. That is why we assume $p_i>2$.
\end{rem}

\begin{theorem}\label{thmmany}
If $p^2|d$ for some odd $d$, then the locus $\mathcal{E}_d(d,d)$ has at least $\sigma(p)$ irreducible components, where $\sigma(p)$ is the number of divisors of $p$.
\end{theorem}
\begin{proof}
A general surface $A\in \mathcal{E}_d(d,d)$ contains precisely two elliptic curves $E,F$, so one can consider a (rational) map $A\mapsto E\cap F\in A[d]$ that is continuous. In particular, the groups $E\cap F$ are abstractly isomorphic for any general abelian surface in the same irreducible component of $\mathcal{E}_d(d,d)$. 

On the other hand, one can easily check that in Lemma \ref{lemany} the surfaces $(E\times F)/G$ have the property that $E\cap F\simeq G$ that are non-isomorphic to each other for various $k|p$.
\end{proof}

\subsection{A natural question}
We have shown a lower bound for the number of irreducible components of $\mathcal{E}_d(d,d)$.
Let us observe the following fact.
\begin{proposition}\label{propint}
Let $A$ be an abelian surface and $E,F\subset A$ elliptic curves. The only possible groups that appear as intersections of $E\cap F$ are of the form $\Z_q\oplus \Z_{q'}$ for $q|q'$. 
\end{proposition}
\begin{proof}
Let $f:E\times F\ra A$ be an addition isogeny. Then $f(E)\cap f(F)\simeq ker(f)$. It is because $f(x,0)=f(0,-y)\iff f(x,y)=0$.

Now, the result follows from the fact that $ker(f)$ is an allowable subgroup of $E[k]\times F[k]\simeq \Z_k^2\oplus\Z_k^2$, hence it cannot be of the form $\Z_q\oplus \Z_{q'}\oplus \Z_{q''}$ for $q|q', q'|q''$.
\end{proof}

Proposition \ref{propint} shows that the intersection of elliptic curves should be of the form $\Z_k\otimes \Z_{\frac dk}$ and all such groups can appear. We have shown that if $d$ is square free, then there is one possible intersection group and only one irreducible component. Therefore one can ask whether fixing an intersection group gives in fact an irreducible component of $\mathcal{E}_d(d,d)$ in general, hence we state the following question.

\begin{question}\label{qirr}
Does the number of irreducible components of $\mathcal{E}_d(d,d)$ equal $\sigma(p)$, where $p$ is the maximal odd number satisfying $p^2|d$?  
\end{question}

\section{Equations for $\mathcal{E}_d(m,n)$ for square free $d$}\label{eqations}
By Theorem \ref{squarefreeirr}, for a square free $d$, the moduli space $\mathcal{E}_d(m,n)$ is a surface in $\mathcal{A}_d$, so one can compute a possible equation for it in the Siegel space in the following way:

Using \cite[Prop 4.4]{ALR} and \cite[Prop 5.4]{K} one sees that preimages of $\mathcal{E}_d$ in $$\mathfrak{h}_2=\left\{Z=\left[\begin{array}{cc}z_1&z_2  \\
    z_2 &z_3 
\end{array}\right]: Im(Z)>0\right\}$$ satisfy singular relations $da_1z_1+a_2z_2+a_3z_3+a_4(z_2^2-z_1z_3)+da_5=0$ for some integers $(a_1,a_2,a_3,a_4,a_5)$ without a common divisor satisfying $\Delta=a_2^2-4da_1a_3-4da_4a_5=p^2$ for some integer $p$.
Note that \cite[Chapter IX]{vdG} uses slightly different notation that results in a different formula.

The equations come from the fact that being non-simple is invariant under isogenies. Hence, one can consider a map that takes an extended period matrix (i.e. the matrix whose columns consist of a symplectic basic for the lattice) $[Z\ D]$ to $[Z\ id_2]$ and use the pullback of a differential form to get equations of the desired form.  

Before proceeding, we would like to set up a (usual) convention that $lcm(0,t)=0$, $gcd(0,t)=t$ and more importantly $\frac{lcm(0,t)}{0}=1$ because $=\frac{lcm(0,t)}{0}=\frac{t}{gcd(0,t)}$.

We start with the following lemma.
\begin{lemma}\label{lemxy} Fix $x,y\in\Z$ satisfying $x\neq y$ and $d|xy$. 
Let $A_{x,y,z_1,z_2}=\C^2/\Lambda$ with $\Lambda$ generated by the columns of $$\left[\begin{array}{cccc}z_1&z_2&1&0  \\
    z_2 &(x+y)z_2-xyz_1&0&d 
\end{array}\right]$$
for $z_1,z_2$ such that the imaginary part of the period matrix is positive definite. Then there exist elliptic curves $E_x,\ E_y$ that are embedded in $A_{x,y,z_1,z_2}$ via maps of analytic representations $[1\ x]$ with the primitive lattice vectors $f_2-yf_1, \frac{lcm(d,x)}{x}e_1+lcm(d,x)e_2$ and $[1\ y]$ with the primitive lattice vectors $f_2-xf_1, \frac{lcm(d,y)}ye_1+lcm(d,y)e_2$. In particular, the exponents are equal to $|lcm(d,x)-\frac{ylcm(d,x)}{x}|$ and $|lcm(d,y)-\frac{xlcm(d,y)}{y}|$.
\end{lemma}
\begin{proof}
By symmetry, it is only necessary to do the case $E_y$. A direct computation shows that $[1\ y]^t(z_2-xz_1)=f_2-xf_1$ which is a primitive vector in $\Lambda$. A vector $\frac{lcm(d,y)}ye_1+lcm(d,y)e_2$ is indeed primitive in $\Lambda$ because the coefficients $\frac{lcm(d,y)}y, \frac{lcm(d,y)}d$ are coprime. The exponent is the absolute value of a symplectic form on primitive vectors in the image, hence we get the assertion. 
\end{proof}
\begin{rem}
Note that for general $z_1,z_2$, the curves are non-isogenous to each other. In such a case $A$ contains only these two elliptic curves, so they have to be complementary to each other. One may also explicitly compute the norm maps and associated idempotents using analytic representations.
\end{rem}

If we multiply the exponents of $E_x,\ E_y$ by $gcd(d,x)$ and $gcd(d,y)$ respectively, we get the same number $dx-dy$ (here, we assume $x>y$). The following lemma is then crucial to find equations.
\begin{lemma}\label{xy}
For any pair $m,n$ of complementary exponents there exist $x,y\in\Z$ (with $d|xy$) satisfying
$$\begin{cases}
dx-dy=m\cdot gcd(d,x)\\
dx-dy=n\cdot gcd(d,y)
\end{cases}$$
\end{lemma}
\begin{proof}
Let $m=\frac{cd}{a},\ n=\frac{cd}{b}$ and assume $a,b,c$ are pairwise coprime.
By Bezout's identity, there exist $u,v$ such that $au-bv=c$. Let $x=\frac{du}{b},\ y=\frac{dv}{a}$. Note that $b,u$ and $a,v$ have to be coprime, so $gcd(d,x)=\frac db$ and $gcd(d,y)=\frac da$. Therefore $dx-dy=\frac{d^2}{ab}(au-bv)=\frac{cd^2}{ab}=m\frac{d}{b}=n\frac{d}{a}$ and we get the assertion.
\end{proof}
To summarize, we have proved the following theorem.
\begin{theorem}\label{periods}
 For any $d,m,n$, we can do the following construction.
 \begin{itemize}
     \item find coprime $a,b,c$ such that $m=\frac{cd}{a},\ n=\frac{cd}{b}$,
     \item find $u,v$ such that $au-bv=c$,
     \item define $x=\frac{du}{b},\ y=\frac{dv}{a}$ (they automatically satisfy $xy|d$).
 \end{itemize}
 Then 
 the set of matrices satisfying $z_3=(x+y)z_2-xyz_1$ with $\Delta=(\frac{cd}{ab})^2$ maps into an irreducible component of $\mathcal{E}_d(m,n)$. 
\end{theorem}
\begin{rem}
If $\mathcal{E}_d(m,n)$ is irreducible then we found one of Humbert surfaces that lies above $\mathcal{E}_d(m,n)$. The equation $z_3=(x+y)z_2-xyz_1$ can be seen as a normalized one, similarly to the principally polarised case considered in \cite[Prop 4.5]{BW} or \cite[Thm 4.(6)]{Bor}.

Moreover, having explicit periods $\tau_E,\tau_F$  of elliptic curves $E,\ F$, one can solve a system of linear equations to get explicit $z_1,z_2$ and hence an explicit period matrix of a surface $A_{x,y,z_1,z_2}$ containing $E$ and $F$.
\end{rem}
\begin{rem}
Note that one can discard the assumption that $(a,b,c)$ are pairwise coprime in the construction given by Theorem \ref{periods} as long as we have $xy|d$. For example, for $d=9$, we have taken $a=b=c=u=1,\ v=0$ and $a=b=c=3, \ u=2,\ v=1$ to get Examples \ref{09} and \ref{36}. 

Since, unfortunately, we do not know how many irreducible components there are for a general $d$, we cannot claim that by considering all possible pairs $x,y$ we get all components of $\mathcal{E}_d(m,n)$.
\end{rem}

At the end of this section, we would like to compare our results to the known ones.
\begin{rem}
In \cite[Thm 9.2.4]{vdG}, the author computes the number of irreducible components of the locus with fixed discriminant. Example \ref{09} (with $\Delta=9^2$) and Example \ref{36} (with $\Delta=3^2$) show that surfaces with different discriminants can lead to irreducible components of the same $\mathcal{E}_d(m,n)$. On the other hand, polarised products having $\Delta=1$ yield components of different $\mathcal{E}_d(m,n)$. Therefore, one cannot conclude our result from \cite{vdG}.

In \cite{G} the author constructs abstract models of irreducible components of $\mathcal{E}_d(m,n)$ indexed by classes of matrices (assuming they are non-empty). 
We can answer his questions (see \cite[Ch.9]{G}) using his notation in the following way. Theorem \ref{irreducible} shows a necessary and sufficient condition for a matrix $M$ to exists. Moreover, using Theorem \ref{irreducible} we have shown that for square free polarisation type $E$, the group $\Gamma_\delta$ acts transitively and using Theorem \ref{periods} one can find a simple form of $M$.  
\end{rem}

\section{Application to genus 3 curves}\label{appl}
Let us call a curve \textit{multielliptic} if its Jacobian is totally decomposable, i.e it is isogenous to a product of elliptic curves. It is still an open question whether in every genus there exists a multielliptic curve (see for example \cite{ES}, \cite{PR}).

For genus 2 and 3, since a (very) general principally polarised abelian variety is a Jacobian, there are infinitely many families of multielliptic curves. Assuming we have a multielliptic curve, a natural question is to ask what the degrees of the elliptic coverings are. For genus 2, we know that a general multielliptic curve is a covering of two elliptic curves and the degrees coincide. The locus of multielliptic curves equals the locus of Humbert surfaces of square discriminant not equal to 1.

For genus 3, multielliptic curves appear in some constructions. Examples of curves that are double coverings of three different elliptic curves may be found in \cite{HLP} or \cite{BO}.  

From now on, $C$ will be a multielliptic genus 3 curve and $E$ will stand for an elliptic curve. In this section we would like to answer the question what triples can be obtained as possible degrees of coverings $f_i:C\ra E_i$ ($i=1,2,3$). As usual, we will only consider \textit{indecomposable} coverings, i.e. coverings $f:C\ra E$ that cannot be decomposed as $f:C\ra E'\ra E$ for some other elliptic curve $E'$. Moreover, if $E_i\ra JC$ are embeddings then we write $JC=E_1\boxplus E_2\boxplus E_3$ if and only if their associated symmetric idempotents satisfy $\epsilon_1+\epsilon_2+\epsilon_3=1$ (i.e. $E_1,e_2,E_3$ are complementary). 

\begin{lemma}
A map $f:C\ra E$ is indecomposable if and only if $f^*:E\ra JC$ is an embedding. In such a case we have $deg(f)=exp(f^*E)$.
\end{lemma}
\begin{proof}
It is a straightforward application of  \cite[Proposition 11.4.3]{BL} using the Hurwitz formula.
\end{proof}
The following theorem is a direct corollary of Theorem \ref{irreducible}.
\begin{theorem}\label{genus3}
Let $C$ be a multielliptic curve such that $JC=E_1\boxplus E_2\boxplus E_3$. Let $f_i:C\ra E_i$ be respective coverings with degrees $d_i$. Then the $d_i$ satisfy Equation \ref{sufficient} (see Proposition \ref{propmainEF}).

Moreover, for every triple $d_1,d_2,d_3$ satisfying Equation \ref{sufficient} (with the assumption that $d_i>1$), there exists a multielliptic genus 3 curve with elliptic curves of exponents $d_i$'s.
\end{theorem}
\begin{proof}
For the first part, it is enough to note that $E_2\boxplus E_3$ is the complementary surface to $E_1$ and therefore it is of type $(1,d_1)$.

For the second part, it is enough to construct such a threefold. If we fix $d_1$, there exists a $(1,d_1)$ polarised surface with $E_2,E_3$ of exponents $d_2,d_3$. Then one uses 3rd remark from page 1548 of \cite{Bor} to construct a principally polarised threefold having $E_1,E_2,E_3$ as complementary elliptic curves with exponents $d_1,d_2,d_3$.
Since $d_i>1$, the threefold is not a polarised product of these curves, hence by \cite[Cor  11.8.2]{BL} it is a Jacobian of a smooth genus 3 curve.
\end{proof}
\begin{rem}
We would like to conclude this section with a few remarks
\begin{enumerate}
    \item By \cite[Thm 4.3.1]{BL} and \cite[Cor 11.8.2]{BL} we can extend the correspondence by considering $d_i=1$ on one side and reducible stable curves of compact type on the other.
\item Note that even though Equation $\ref{sufficient}$ seems not to be symmetric with $d$ playing a different role than $m,n$, Theorem \ref{genus3} shows that it is fully symmetric.
In fact, Equation \ref{sufficient} can be rephrased as two (symmetric) equations $$lcm(m,n)=lcm(m,d)=lcm(n,d).$$
\item Using Section \ref{eqations} and the proof of Theorem \ref{genus3} one may be able to write down explicit period matrices of Jacobians of multielliptic genus 3 curves for triples $d_1,d_2,d_3$.  
\end{enumerate}
\end{rem}

\end{document}